\documentclass[a4, 11pt]{amsart}

\usepackage{graphicx}
\usepackage{cite}
\usepackage[utf8]{inputenc}
\usepackage[english]{babel}
\usepackage{amsmath}
\usepackage{amsfonts}
\usepackage{amssymb}
\usepackage{xcolor}
\usepackage{amsthm}
\usepackage{xypic}
\usepackage{enumitem}

\newtheorem{thm}{Theorem}[section]
\newtheorem{lemma}[thm]{Lemma}
\newtheorem{prop}[thm]{Proposition}
\newtheorem{cor}[thm]{Corollary}

\theoremstyle{definition}
\newtheorem{defn}[thm]{Definition}
\newtheorem{example}[thm]{Example}

\theoremstyle{remark}
\newtheorem{rem}[thm]{Remark}

\numberwithin{equation}{section}

\DeclareMathOperator{\Ker}{Ker}
\DeclareMathOperator{\Img}{Im}
\DeclareMathOperator{\Mat}{Mat}
\DeclareMathOperator{\Hom}{Hom}
\DeclareMathOperator{\End}{End}

\DeclareMathOperator{\Rad}{Rad}
\DeclareMathOperator{\I}{I}

\newcommand{\sm}{\sigma[M]}

\newcommand{\ess}{\leq^\text{ess}}
\newcommand{\dleq}{\leq^\oplus}

\begin{document}

\title{Abelian Endoregular Modules}

\author{Mauricio Medina-B\'arcenas}
\address{Department of Mathematics, Chungnam National University,
Yuseong-gu Daejeon 34134, Republic of Korea}
%\curraddr{Department of Mathematics and Statistics,
%Case Western Reserve University, Cleveland, Ohio 43403}
\email{mmedina@cnu.ac.kr, corresponding author}

\author{Hanna Sim}
%\address{Department of Mathematics, Chungnam National University,
%Yuseong-gu Daejeon 34134, Republic of Korea}
\email{hnsim@cnu.ac.kr}

%    General info

%\today

%\date



\begin{abstract}
In this paper, we introduce the notion of abelian endoregular modules as those modules whose endomorphism rings are abelian von Neumann regular.
We characterize an abelian endoregular module $M$ in terms of its $M$-generated submodules.
We prove that if $M$ is an abelian endoregular module then so is every $M$-generated submodule of $M$.
Also, the case when the (quasi-)injective hull of a module as well as the case when a direct sum of modules is abelian endoregular are presented.
At the end, we study abelian endoregular modules as subdirect products of simple modules.
\end{abstract}

\maketitle

\section*{Introduction}

%\textcolor{blue}{We have to see what else we can add to the introduction and look for typos and bad sentences.}

\emph{Von Neumann Regular rings} have been an active topic in ring theory since they were introduced. An intrinsic study of these rings can be found, in full detail in Goodearl's book \cite{goodearlneumann}. Among von Neumann regular rings there are subclasses of interest in their own. Some of such classes are the \emph{unit regular rings}  and the \emph{strongly regular rings}. Recall that a ring $R$ is von Neumann regular if for any $r\in R$ there exists $x\in R$ such that $rxr=r$; if $x$ can be chosen as a unit then $R$ is called unit regular; and if $r^2x=r$ then $R$ is called strongly regular. Strongly regular rings where named by Arens and Kaplansky \cite{arenstopological}. They proved that a strongly regular ring is von Neumann regular an every one sided ideal is two sided. Forsythe and McCoy \cite{forsytheonthecomm} showed that a strongly regular ring is a von Neumann regular ring such that all idempotents are central. Recall that a ring $R$ is called \emph{abelian} provided that every idempotent in $R$ is central. In \cite[Theorem 3.5]{goodearlneumann} can be found a proof that a ring is strongly regular if and only if it is an abelian von Neumann regular ring. Azumaya in 1948 characterized modules whose endomorphism rings are von Neumann regular as those modules $M$ such that $\Ker\varphi\dleq M$ and $\Img\varphi\dleq M$ for all $\varphi\in\End_R(M)$ \cite{azumayageneralized}. Recently, these kind of modules have been studied in detail, such modules are called \emph{endoregular}, and many properties are discussed in \cite{leemodules}. Later in \cite{leeunit} the modules whose endomorphism rings are unit regular rings were presented. Now, in this manuscript, we deal with the analogous case for abelian regular rings. We call a module $M$ \emph{abelian endoregular} if $\End_R(M)$ is an abelian regular ring. We study an abelian endoregular module $M$, looking at its $M$-generated submodules. We prove that direct summands of abelian endoregular modules inherit the property (Corollary \ref{dirsummand}), moreover if $M$ is an abelian endoregular module then every $M$-generated submodule of $M$ inherits the property (Theorem \ref{limit}). Also, we characterize direct sums of abelian endoregular modules (Proposition \ref{sums}). In the ring case, an abelian regular ring is characterized as a regular ring which is a subdirect product of division rings, here we study the analogous module-theoretic result. If $M$ is a quasi-duo endoregular module which is a subdirect product of simple modules then $M$ is abelian endoregular (Proposition \ref{sdpse}). In order to find the converse, we make use of the concept of (resp., \emph{semi})\emph{prime submodule} introduced by Raggi et.al., in (resp., \cite{raggisemiprime}) \cite{raggiprime} with the product of submodules defined in \cite{BicanPr}.

This paper is divided as follows: Section 1 concerns to the preliminary background to make this work self-contained as much as possible. In these preliminaries we recall the definition of a product of submodules given in \cite{BicanPr} and the definitions of prime and semiprime submodules given in \cite{raggiprime} and \cite{raggisemiprime} respectively. Looking for examples to illustrate and delimit our results we make a brief exposition of \emph{incidence algebras}. We prove that given a module $M$ with some properties over a ring $R$, the endomorphism ring of $M$ over $R$ and the endomorphism ring of $M^{(X)}$ as module over the incidence algebra $\I(X,R)$ are isomorphic, where $X$ denotes a preordered set (Proposition \ref{incend}). In Section 2 we extend some results which are presented in \cite[Ch. 3]{goodearlneumann} to the module case.  We characterize an abelian endoregular module $M$ as an endoregular module such that every $M$-generated submodule is fully invariant (Proposition \ref{subfi}). Also, we characterize the direct sums of abelian endoregular modules (Proposition \ref{sums}). At the end of Section 2 we present the case when the (quasi-)injective hull of a module is abelian endoregular (Theorem \ref{injhullsr}). Section 3 presents abelian endoregular modules as subdirect products of simple modules. We show when a subdirect product of simple modules is abelian endoregular and vise versa (Proposition \ref{sdpse} and Theorem \ref{stqdr}). We finish the paper with a connection between abelian endoregular modules and cosemisimple modules (Proposition \ref{coss}).

\section{Preliminaries}

Throughout the paper $R$ denotes an associative ring with identity. All $R$-modules are unitary right $R$-modules.  For a submodule $N$ of a module $M$ we use the notation $N\leq M$ and if $N$ is a direct summand we write $N\dleq M$. If $X$ is a set, $M^{(X)}$ means the direct sum of $|X|$ copies of $M$; if $X$ is finite, say $|X|=n$ we write $M^{(n)}$. Recall that a submodule $N$ of $M$ is \emph{fully invariant} ($N\leq_{fi}M$) if $\varphi(N)\subseteq N$ for all $\varphi\in\End_R(M)$.

The next lemma will be used without mention in many results along the paper.

\begin{lemma}\label{episumdir}
Let $\rho:M\to N$ be an epimorphism and $L\leq N$. If $\rho^{-1}(L)\dleq M$ then $L$ is a direct summand of $N$.
\end{lemma}

\begin{proof}
Suppose $M=\rho^{-1}(L)\oplus K$. Since $\rho$ is an epimorphism $N=L+\rho(K)$. If $x\in L\cap\rho(K)$, then there exists $k\in K$ such that $x=\rho(k)$ and since $x\in L$, $k$ lies in $\rho^{-1}(L)$. Thus $k\in\rho^{-1}(L)\cap K=0$, that implies $x=0$. Therefore $N=L\oplus\rho(K)$.
\end{proof}

%\begin{flushleft}
%Let $M$ and $N$ be modules. If $S=\End_R(M)$ then $\Hom_R(M,N)$ is a right $S$-module. In particular, if $N\leq M$ then $\Hom_R(M,N)$ is a right ideal of $S$. Given an ideal $I$ of $S$,
%\[IM=\sum\{f(M)\mid f\in I\}\]
%that is, the elements of $IM$ are finite sums $\sum_{i=1}^nf_i(m_i)$ with $f_i\in I$ and $m_i\in M$.
%\end{flushleft}

Recall that a module $M$ is said to be \emph{endoregular} \cite{leemodules} if $\End_R(M)$ is a von Neumann regular ring.

\begin{prop}[Theorem 16, \cite{azumayageneralized}]\label{azu}
Let $M$ be a module and $S=\End_R(M)$. Let $\varphi\in S$ be any element. Then, $\Ker\varphi\dleq M$ and $\Img\varphi\dleq M$ if and only if there exists $\psi\in S$ such that $\varphi\psi\varphi=\varphi$. Consequently, $M$ is endoregular if and only if $\Ker\varphi$ and $\Img\varphi$ are direct summands of $M$.
\end{prop}

%The following proposition will be useful along the paper.

\begin{prop}\label{kerdirecsum}
Let $M$ be an endoregular module.
\begin{enumerate}[label=\emph{(\arabic*)}]
\item If $\varphi:M^{(n)}\to M^{(\ell)}$ is a homomorphism with $n,\ell>0$ then, $\Ker\varphi\dleq M^{(n)}$ and $\Img\varphi\dleq M^{(\ell)}$.
\item For any direct summand $M'\dleq M$, any submodule $N\leq M$ and any homomorphism $\varphi:M'\to N$, $Ker\varphi\dleq M'$ and $\Img\varphi\dleq N$.
\end{enumerate}\end{prop}

\begin{proof}
It follows from \cite[Proposition 3.4, Theorem 3.6 and Corollary 3.15]{leemodules}
\end{proof}

A module $M$ is said to have \emph{summand sum property} (SSP), if the sum of any two direct summands is a direct summand of $M$. The module $M$ is said to have \emph{summand intersection property} (SIP), if the intersection of two direct summands is a direct summand of $M$.

\begin{lemma}\label{endosip}
If $M$ is an endoregular module then $M$ has SSP and SIP.
\end{lemma}

\begin{proof}
In \cite[Lemma 2.1]{garciaproperties} it is proved if given idempotents $e,f\in\End_R(M)$ such that $\Img(ef)\dleq M$ (resp., $\Ker(ef)$) then $M$ satisfies SSP (resp., SIP).
\end{proof}

\begin{defn}
Let $M$ be a module. A module $N$ is said to be (\emph{finitely}) \emph{$M$-generated} if there exists an epimorphism $M^{(\mathcal{I})}\to N$ for some index set $\mathcal{I}$ (resp., $M^{(n)}\to N$ for some $n>0$).
\end{defn}

\begin{prop}\label{finmsumand}
Let $M$ be an endoregular module and $N\leq M$. Then, $N$ is finitely $M$-generated if and only if $N$ is a direct summand of $M$.
\end{prop}

\begin{proof}
Suppose that $N\leq M$ is finitely $M$-generated. Hence, there exists an epimorphism $\rho:M^{(n)}\to N$ for some $n>0$. Thus, $N=\Img\rho\dleq M$ by Proposition \ref{kerdirecsum}. The converse is clear.
\end{proof}

Since we are interested in the endomorphism ring of a given module, we will present a short exposition on incidence algebras which will help us to find examples which satisfy our desire conditions. For the notation and terminology, the reader is referred to \cite{feinbergcharacterization} and \cite{spiegelincidence}.

\begin{defn}
Given $x$ and $z$ in a preordered set $(X,\mathcal{R})$, the \emph{interval or segment from $x$ to $z$ is} $\{y\in X\mid x\mathcal{R}y\mathcal{R}z\}$ and is denoted by $[x,z]$. A preordered set $(X,\mathcal{R})$ is \emph{locally finite} if every segment of $X$ is finite.
\end{defn}

\begin{defn}
The \emph{incidence algebra $\I(X,A)$} of the locally finite preordered set $(X,\mathcal{R})$ over the commutative ring $A$ with identity  is
\[\I(X,A)=\{f:X\times X\to A\mid f(x,y)=0\text{ if } (x,y)\notin \mathcal{R}\}\]
with operations given by
\[(f+g)(x,y)=f(x,y)+g(x,y)\]
\[(fg)(x,y)=\sum_{x\mathcal{R}z\mathcal{R}y} f(x,z)g(z,y)\]
\[(r\cdot f)(x,y)=rf(x,y)\]
for $f,g\in \I(X,A)$ with $r\in A$ and $x,y,z\in X$.
\end{defn}

\begin{defn}
A preordered set $(X,\mathcal{R})$ is \emph{upper finite} if the set $\{y\in X\mid x\mathcal{R} y\}$ is finite for every $x\in X$.
\end{defn}

Let $(X,\mathcal{R})$ be an upper finite preordered set  and $A$ be a commutative ring with identity. For every right $A$-module $M_A$, we consider $M(X)=M^{(X)}$ as abelian group. If $R=\I(X,A)$ is the incidence algebra of $X$ over the ring $A$ then, $M(X)$ is a right $R$-module with action
\[(m_x)_{x\in X}f=\sum_{x\in X}(m_xf(x,y))_{y\in X},\]
for $(m_x)_{x\in X}\in M(X)$ and $f\in R$.

\noindent\textbf{Notation:} For $x\mathcal{R} y$, denote by $e_{xy}\in R$ the element in $\I(X,A)$ such that $e_{xy}(x,y)=1$ and $0$ otherwise.

\begin{prop}\label{incend}
Let $(X,\mathcal{R})$ be an upper finite preordered set with an element $w\in X$ such that $w\mathcal{R}x$ for all $x\in X$. Let $R=\I(X,A)$ be the incidence algebra of $X$ over the commutative ring with identity $A$ and let $M_A$ be a right $A$-module. If one of the following conditions is satisfied,
\begin{enumerate}[label=\emph{(\arabic*)}]
\item $M_A$ is cyclic.
\item $M_A=AS^{-1}$, the ring of fractions of $A$ with respect to a multiplicative closed set $S$.
\end{enumerate}
Then $\End_A(M)\cong\End_R(M(X))$.
\end{prop}

\begin{proof}
Let $\delta_{xy}$ denote the Kronecker's delta. Consider $m\in M$ and $w\in X$ such that $w\mathcal{R}x$ for all $x\in X$. Let $\Phi\in\End_R(M(X))$ be any element. Set $\Phi((\delta_{wx}m)_{x\in X})=(n_x)_{x\in X}\in M(X)$. Then,
\begin{equation}\label{eq1}
\Phi((\delta_{wx}m)_{x\in X})=\Phi((\delta_{wx}m)_{x\in X}e_{ww})=\Phi((\delta_{wx}m)_{x\in X})e_{ww}=(n_x)_{x\in X}e_{ww}=(\delta_{wx}n_w)_{x\in X}.
\end{equation}

Since $w\mathcal{R} y$ for all $y\in X$, we have
\begin{equation}\label{eq2}
\Phi((\delta_{yx}m)_{x\in X})=\Phi((\delta_{wx}m)_{x\in X}e_{wy})=\Phi((\delta_{wx}m)_{x\in X})e_{wy}=(\delta_{wx}n_w)_{x\in X}e_{wy}=(\delta_{yx}n_w)_{x\in X}.
\end{equation}

For (1), suppose $M_A=mA$ for some $m\in M$. Let $(m_x)_{x\in X}\in M(X)$ and $\Phi\in\End_R(M(X))$ be any endomorphism. Then, $(m_x)_{x\in X}=(ma_x)_{x\in X}$ for some $a_x\in A$ and by (\ref{eq2}), for all $y\in X$, $\Phi((\delta_{yx}m)_{x\in X})=(\delta_{yx}n)_{x\in X}$ for some $n\in M$.  Hence,
\[\Phi((ma_x)_{x\in X})
=\sum_{x\in X}\Phi((\delta_{xy}m)_{y\in X}e_{xx}a_x)
 =\sum_{x\in X}\Phi((\delta_{xy}m)_{y\in X})e_{xx}a_x\]
 \[ =\sum_{x\in X}(\delta_{xy}n)_{y\in X}e_{xx}a_x
 =(na_x)_{x\in X}.\]

Define $\varphi\in\End_A(M)$ such that $\varphi(m)=n$, then $\Phi((m_x)_{x\in X})=(\varphi(m_x))_{x\in X}$. It is clear that for given $\varphi\in\End_A(M)$, we have an endomorphism $\Phi\in\End_R(M(X))$ defined as $\Phi((m_x)_{x\in X})=(\varphi(m_x))_{x\in X}$. Thus $\End_A(M)\cong\End_R(M(X))$.

For (2), suppose $M_A=AS^{-1}$ for some multiplicative closed set $S\subseteq A$. Let $\Phi\in\End_R(M(X))$ be any endomorphism. By (\ref{eq1}) and (\ref{eq2}), $\Phi((\delta_{wx}1)_{x\in X})=(\delta_{wx}q)_{x\in X}$ for some $q\in M$ and $\Phi((\delta_{yx}1)_{x\in X})=(\delta_{yx}q)_{x\in X}$ for every $y\in X$. Let $b\in S$. Set $\Phi((\delta_{wx}b^{-1})_{x\in X})=(n_x)_{x\in X}\in M(X)$. By (\ref{eq1}) $\Phi((\delta_{wx}b^{-1})_{x\in X})=(\delta_{wx}n_w)_{x\in X}$ then,
\[(\delta_{wx}q)_{x\in X}=\Phi((\delta_{wx}1)_{x\in X}) =\Phi((\delta_{wx}b^{-1}b)_{x\in X})
 =\Phi((\delta_{wx}b^{-1})_{x\in X}e_{ww}b)
 =\Phi((\delta_{wx}b^{-1})_{x\in X})e_{ww}b\]
\[ =(\delta_{wx}n_w)_{x\in X}e_{ww}b
 =(\delta_{wx}n_wb)_{x\in X}.\]

Hence, $b^{-1}q=n_w$. Note that, for $w\neq x\in X$
\[(\delta_{xy}n_x)_{y\in X}=(n_x)_{x\in X}e_{xx}=\Phi((\delta_{wy}b^{-1})_{y\in X}e_{xx})=0.\]
Thus, $\Phi((\delta_{wx}b^{-1})_{x\in X})=(\delta_{wx}b^{-1}q)_{x\in X}$. Let $(m_x)_{x\in X}$ be any element in $M(X)$. Then $m_x=a_xb_x^{-1}$ for some $a_x\in A$ and $b_x\in S$. Therefore,

\[\Phi((m_x)_{x\in X}) =\sum_{x\in X}\Phi((\delta_{xy}m_x)_{y\in X})
 =\sum_{x\in X}\Phi((\delta_{xy}b_x^{-1})_{y\in X}e_{xx}a_x)
 =\sum_{x\in X}\Phi((\delta_{wy}b_x^{-1})_{y\in X})e_{wx}e_{xx}a_x\]
\[ =\sum_{x\in X}(\delta_{wy}b_x^{-1}q)_{y\in X}e_{wx}e_{xx}a_x
 =(qm_x)_{x\in X}.\]

Hence $\Phi$ is completely determined by $q$. Thus, $\End_R(M(X))\cong AS^{-1}\cong\End_A(M)$.
\end{proof}

What follows will be needed for Section 3. We will recall a product of submodules of a given module and some concepts related with this product such as prime and semiprime submodules.

In \cite{BicanPr} was introduced a product of modules. Let $M$ and $N$ be modules. Given $K\leq M$, \emph{the product of $N$ and $K$} is defined as
\[N_MK=\sum\{f(K)\mid f\in \Hom_R(M,N)\}.\]

We can see that if $M$ is an $R$-module and $I$ is a right ideal of $R$ then $MI=M_RI$. For more properties of this product see \cite[Proposition 1.3]{PepeGab}.
%\emph{The annihilator of $N$ in $M$} \cite{beachy2002m} is defined as
%\[\ann_M(N)=\bigcap\{\Ker f|f\in\Hom_R(M,N)\}.\]
%
%It is clear that $\ann_M(N)$ is the largest submodule of $M$ such that
%\[N_M\ann_M(N)=0.\]

\begin{defn}\label{defsemiprime}
A proper fully invariant submodule $N$ of $M$ is called \emph{prime} in $M$ \cite{raggiprime} if $K_M L\subseteq N$, then $K\subseteq N$ or $L\subseteq N$ for any fully invariant submodules $K,L$ of $M$. The submodule $N$ is called \emph{semiprime} in $M$ \cite{raggisemiprime} if $K_M K\subseteq N$, then $K\subseteq N$ for any fully invariant submodule $K$ of $M$. The module $M$ is called a (resp., \emph{semi})\emph{prime module} if $0$ is a (resp., semi)prime submodule in $M$. The set of all prime submodules of $M$ is denoted by $Spec(M)$.
\end{defn}

Recall that a module $M$ is said to be \emph{quasi-projective} ($M$-\emph{projective}) if for every epimorphism $\pi:M\to N$ and every homomorphism $f:M\to N$ there exists an endomorphism $g:M\to M$ such that $\pi g=f$. Analogously, $M$ is $M$\emph{-injective} if for every monomorphism $\eta:N\to M$ and every homomorphism $f:N\to M$ there exists an endomorphism $g:M\to M$ such that $g\eta=f$.

\begin{prop}\label{2}
Let $M$ be a module and $N$ a proper fully invariant submodule of $M$. Then,
\begin{enumerate}[label=\emph{(\arabic*)}]
\item If $N$ is (resp., semi)prime in $M$ then $M/N$ is a (resp., semi)prime module.
\item If $M$ is quasi-projective and $M/N$ is a (resp., semi)prime module then $N$ is (resp., semi)prime in $M$.
\end{enumerate}
\end{prop}

\begin{proof}
\cite[Proposition 18]{raggiprime} and \cite[Proposition 13]{raggisemiprime}.
\end{proof}

\begin{lemma}\label{maxprime}
Let $M$ be quasi-projective. If $N$ is fully invariant such that $N$ is maximal in the lattice of fully invariant submodules of $M$ then $N$ is prime in $M$.
\end{lemma}

\begin{proof}
Since $N$ is maximal in the lattice of fully invariant submodules of $M$, then $M/N$ has no nontrivial fully invariant submodules. It follows that $M/N$ is a prime module. By Proposition \ref{2}, $N$ is prime in $M$.
\end{proof}

\section{Abelian Endoregular Modules}

\begin{defn}
An $R$-module $M$ is called \emph{abelian endoregular} if $\End_R(M)$ is an abelian von Neumann regular ring.
\end{defn}

\begin{example}
\begin{enumerate}
\item[(i)] $R$ is an abelian von Neumann regular ring if and only if $R_R$ is an abelian endoregular module.
\item[(ii)] Consider the $\mathbb{Z}$-module $\mathbb{Q}_\mathbb{Z}$. We have that $\End_\mathbb{Z}(\mathbb{Q})=\mathbb{Q}$. Hence, $\mathbb{Q}_\mathbb{Z}$ is an abelian endoregular module.
\item[(iii)] Let $K$ be a field. Let $R=\prod_{i=1}^\infty K_i$ with $K_i=K$. Then, the ideal $\bigoplus_{i=1}^\infty K_i$ is an abelian endoregular $R$-module because $R=\End_R(\bigoplus_{i=1}^\infty K_i)$.
\item[(iv)] If $\{S_i\}_{i\in\mathcal{I}}$ is a family of simple modules nonisomorphic in pairs, then $M=\bigoplus_{i\in\mathcal{I}}S_i$ is abelian endoregular (see Proposition \ref{sums}). Moreover $\prod_{i\in\mathcal{I}}S_i$ is an abelian endoregular module. For, note that $\End_R(\prod_{i\in\mathcal{I}}S_i)\cong\prod_{i\in\mathcal{I}}\End_R(S_i)$.
\item[(v)] Let $K$ be a field and $A$ be a hereditary $K$-algebra. Then every indecomposable projective $A$-module is abelian endoregular. It follows from the fact that if $P$ is an indecomposable projective $A$-module then $\End_A(P)=K$ \cite[Ch.VII Corollary 1.5]{assemelements}.
\end{enumerate}
\end{example}

\begin{example}
Let $(X,\leq)$ be the partial ordered set
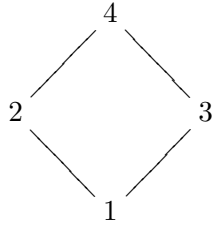
\[\xymatrix{ & 4\ar@{-}[dl]\ar@{-}[dr] & \\ 2\ar@{-}[dr] &  & 3\ar@{-}[dl] \\ & 1 & }\]
Then, $R=\I(X,\mathbb{Z})\cong\left(\begin{smallmatrix}
\mathbb{Z} & \mathbb{Z} & \mathbb{Z} & \mathbb{Z} \\
0 & \mathbb{Z} & 0 & \mathbb{Z} \\
0 & 0 & \mathbb{Z} & \mathbb{Z} \\
0 & 0 & 0 & \mathbb{Z}
\end{smallmatrix}\right)$ (see \cite[Proposition 1.2.7]{spiegelincidence}). Hence, $\End_R(\mathbb{Q}{(4)})\cong\mathbb{Q}$, by Proposition \ref{incend}. Thus, the right $R$-module $\mathbb{Q}{(4)}$ is an abelian endoregular module.
\end{example}

\begin{prop}
Let $(X,\mathcal{R})$ be an upper finite preordered set with an element $w\in X$ such that $w\mathcal{R}x$ for all $x\in X$. Let $R=\I(X,A)$ be the incidence algebra of $X$ over the ring $A$ and let $M_A$ be a cyclic right $A$-module. If $M_A$ is an abelian endoregular module then so is $M(X)_R$.
\end{prop}

\begin{proof}
It follows from Proposition \ref{incend}.
\end{proof}

The following propositions which characterize abelian endoregular modules were proved in \cite{leemodules}.

\begin{prop}[Proposition 2.22, \cite{leemodules}]\label{phikerima}
The following conditions are equivalent for a module $M$:
\begin{enumerate}[label=\emph{(\alph*)}]
\item $M$ is abelian endoregular.
\item $M=\Ker\varphi\oplus\Img\varphi$ for any $\varphi\in\End_R(M)$.
\end{enumerate}
\end{prop}

\begin{cor}[Corollary 2.24, \cite{leemodules}]\label{ephikerima}
The following conditions are equivalent for a module $M$ and $e^2=e\in\End_R(M)$:
\begin{enumerate}[label=\emph{(\alph*)}]
\item $eM$ is abelian endoregular.
\item $M=\Ker(e\varphi e)\oplus\Img(e\varphi e)$ for any $\varphi\in\End_R(M)$.
\end{enumerate}
\end{cor}

\begin{cor}\label{dirsummand}
Let $M$ be an abelian endoregular module. Then every direct summand of $M$ is abelian endoregular.
\end{cor}

\begin{proof}
Let $N\dleq M$. There exists an idempotent $e\in\End_R(M)$ such that $eM=N$. Consider $e\varphi e$ with $\varphi\in\End_R(M)$. Since $M$ is abelian endoregular, $M=\Ker(e\varphi e)\oplus\Img(e\varphi e)$ by Proposition \ref{phikerima}. Hence $eM=N$ is abelian endoregular by Corollary \ref{ephikerima}.
\end{proof}

Given a two-sided ideal $I$ of a ring $R$, it is said $I$ is regular if for any $x\in I$ there exists $y\in I$ such that $xyx=x$. It is known that $R$ is von Neumann regular if and only if $I$ and $R/I$ are regular. It is also known that if $R$ is abelian regular and $I$ is a two-sided ideal of $R$ then $R/I$ is an abelian regular ring. In the next proposition we write down last ideas in the module-theoretic context.

\begin{prop}
Let $M$ be a quasi-projective module and $N$ be a fully invariant submodule of $M$. The following conditions are equivalent:
\begin{enumerate}[label=\emph{(\alph*)}]
\item $M$ is abelian endoregular.
\item $M/N$ is abelian endoregular and for any $\varphi\in\Hom_R(M,N)$, $\Ker\varphi$ and $\Img\varphi$ are direct summands of $M$ and $N=(\Ker\varphi\cap N)\oplus\Img\varphi$ .
\end{enumerate}
\end{prop}

\begin{proof}
Given any $\varphi\in\End_R(M)$, $\varphi$ defines an endomorphism $\overline{\varphi}\in\End_R(M/N)$ because $N$ is fully invariant. Hence we have a ring homomorphism $\Theta:\End_R(M)\to\End_R(M/N)$ given by $\Theta(\varphi)=\overline{\varphi}$. Note that $\Ker\Theta=\Hom_R(M,N)$ and $\Theta$ is surjective because $M$ is quasi-projective.

(a)$\Rightarrow$(b) Let $\varphi:M\to N$ be any homomorphism. By Proposition \ref{phikerima}, $M=\Ker\varphi\oplus\Img\varphi$. It follows that $N=(\Ker\varphi\cap N)\oplus\Img\varphi$. Since $\Theta$ is surjective, $\End_R(M/N)$ is an abelian regular ring. Thus $M/N$ is an abelian regular module.

(b)$\Rightarrow$(a) Let $\varphi\in\Hom_R(M,N)$. Since $\Ker\varphi$ and $\Img\varphi$ are direct summands of $M$, there exists $\psi\in\End_R(M)$ such that $\varphi\psi\varphi=\varphi$ by Proposition \ref{azu}. Note that $\psi\varphi\psi\in\Hom_R(M,N)$ because $N$ is fully invariant in $M$. Hence $\varphi(\psi\varphi\psi)\varphi=\varphi$. This implies that the two-sided ideal $\Hom_R(M,N)$ of $\End_R(M)$ is regular. Since $\End_R(M/N)\cong\End_R(M)/\Hom_R(M,N)$ is a von Neumann regular ring, $\End_R(M)$ is a von Neumann regular ring. Now, let $\varphi\in\End_R(M)$ such that $\varphi^2=0$. Hence $0=\Theta(\varphi^2)=\Theta(\varphi)^2$. Since $\End_R(M/N)$ is abelian regular, $\End_R(M/N)$ has no nonzero nilpotent elements \cite[Theorem 3.2]{goodearlneumann}. Hence $\Theta(\varphi)=0$ and so, $\varphi\in\Hom_R(M,N)$. Therefore $\Img\varphi\subseteq\Ker\varphi\cap N$ because $\varphi^2=0$. By hypothesis we have that $N=(\Ker\varphi\cap N)\oplus \Img\varphi$. Hence, $\Img\varphi=0$. Thus, $\End_R(M)$ is a regular ring with no nonzero nilpotents elements, that is, $\End_R(M)$ is an abelian regular ring by \cite[Theorem 3.2]{goodearlneumann}. Hence $M$ is an abelian endoregular module.
\end{proof}

In the next proposition we characterize an abelian endoregular module $M$ in terms of those submodules which are generated by $M$.

\begin{prop}\label{subfi}
The following conditions are equivalent for a module $M$:
\begin{enumerate}[label=\emph{(\alph*)}]
\item $M$ is abelian endoregular.
\item $M$ is endoregular and every $M$-generated submodule of $M$ is fully invariant.
\end{enumerate}
\end{prop}

\begin{proof}
(a)$\Rightarrow$(b) By hypothesis $M$ is endoregular. Let $N$ be an $M$-generated submodule of $M$ and set $S=\End_R(M)$. Consider the right ideal $\Hom_R(M,N)$ of $S$. Since $S$ is an abelian regular ring, $\Hom_R(M,N)$ is a two-sided ideal by \cite[Theorem 3.2]{goodearlneumann}. This implies that $\Hom_R(M,N)M$ is a fully invariant submodule of $M$. Since $N$ is $M$-generated, $N=\Hom_R(M,N)M$. Thus, $N$ is fully invariant.

(b)$\Rightarrow$(a) Let $e\in S$ be an idempotent. By hypothesis $eM$ is fully invariant. Let $\varphi\in S$. Then $\varphi e(M)\leq eM$. This implies that $e\varphi e=\varphi e$. Hence, $(1-e)\varphi e=(1-e)e\varphi e=0$. Thus $(1-e)Se=0$. By \cite[Lemma 3.1]{goodearlneumann} $e$ is a central idempotent. Therefore, $M$ is abelian endoregular.
\end{proof}

%\begin{example}
%$\mathbb{Q}_\mathbb{Z}$ is an abelian regular $\mathbb{Z}$-module, but no nonzero proper submodule of $\mathbb{Q}_\mathbb{Z}$ is fully invariant. Note that $\mathbb{Q}_\mathbb{Z}$ does not generate its proper nonzero submodules.
%\end{example}

\begin{cor}
Every direct summand of an abelian endoregular module is fully invariant.	
\end{cor}

The following proposition extends \cite[Theorem 3.4]{goodearlneumann}.

\begin{prop}\label{propinc}
The following conditions are equivalent for an endoregular module $M$:
\begin{enumerate}[label=\emph{(\alph*)}]
\item $M$ is an abelian endoregular module.
\item Two isomorphic $M$-generated submodules of $M$ are equal.
\item For an $M$-generated submodule $B$ of $M$ such that $B\oplus B$ embeds in $M$, $B=0$.
\item For two $M$-generated submodules $A,B$ of $M$ such that $A\cap B=0$, $\Hom_R(A,B)=0$.
\item The lattice of direct summands of $M$ is distributive.
\end{enumerate}
\end{prop}

\begin{proof}
(a)$\Rightarrow$(b) Let $N$ and $K$ be two nonzero $M$-generated submodules of $M$ and let $\theta:N\to K$ be an isomorphism. Since $N$ is $M$-generated, $\Hom_R(M,N)\neq 0$. Let $0\neq\varphi:M\to N$. Set $A=\varphi(M)\leq N$ and $B=\theta(A)\leq K$. Hence, $A\cong B$. Since $A$ is finitely $M$-generated, $A\dleq M$ (Proposition \ref{finmsumand}). Therefore, there exists an idempotent $e\in \End_R(M)$ such that $A=eM$. Analogously there exists an idempotent $f\in\End_R(M)$ such that $B=fM$. Note that $(\theta e)M=fM$ and $(\theta^{-1}f)M=eM$. Since $M$ is abelian endoregular, we have
\[f=\theta e=\theta e e=e\theta e=ef\;\text{and;}\]
\[e=\theta^{-1}f=\theta^{-1}f f=f\theta^{-1} f=fe=ef.\]
Thus, $e=f$. This implies $\varphi(M)=A=eM=fM=B\leq K$ and so $N=\Hom_R(M,N)M\leq K$. Analogously, $K\leq N$. Thus $N=K$.

(b)$\Rightarrow$(c) It is clear.

(c)$\Rightarrow$(d) Suppose $\Hom_R(A,B)\neq 0$ and let $0\neq\varphi:A\to B$. Since $A$ is $M$-generated, there exists $\psi:M\to A$ such that $\varphi\psi\neq 0$ and $\psi(M)\dleq M$. We have $M$ is an endoregular module, then by Proposition \ref{kerdirecsum} $\Ker\varphi\cap \psi(M)\dleq\psi(M)$, that is, $\psi(M)=(\Ker\varphi\cap \psi(M))\oplus C$ with $C\cong\varphi\psi(M)$. Thus,
\[C\oplus C\cong C\oplus \varphi\psi(M)\leq A\oplus B\leq M.\]
By hypothesis, $C=0$. A contradiction. Thus $\Hom_R(A,B)=0$.

(d)$\Rightarrow$(a) Set $S=\End_R(M)$ and let $e\in S$ be an idempotent. Then $M=eM\oplus (1-e)M$. We have that
\[\Hom_R(eM,(1-e)M)=(1-e)Se.\]
By hypothesis $\Hom_R(eM,(1-e)M)=0$, hence $e$ is central by \cite[Lemma 3.1]{goodearlneumann}.

(a)$\Leftrightarrow$(e) The proof of $(a)\Leftrightarrow(e)$ in \cite[Theorem 3.4]{goodearlneumann} works here.
\end{proof}

Recall that a module $M$ is said to be \emph{unit endoregular} if $\End_R(M)$ is a unit regular ring. In \cite{leeunit} were introduce these modules and many properties were presented. It is clear that every abelian endoregular module is unit endoregular since abelian regular rings are unit regular. The next propositions provide a partial converse.

\begin{prop}
Let $M$ be a unit endoregular module.
If $M=\Img\varphi+\Ker\varphi$ for any $\varphi \in \End_R(M)$, then $M$ is an abelian endoregular module.
\end{prop}

\begin{proof}
By \cite[Theorem 16]{leeunit} and hypothesis for any $\varphi \in \End_R(M)$,
 $$\Img\varphi \cap \Ker\varphi \cong M/(\Img\varphi+\Ker\varphi) \cong 0.$$
Hence $M$ is an abelian endoregular module by Proposition \ref{phikerima}.
\end{proof}

\begin{prop}
Let $M$ be a unit endoregular module. If any idempotent in $\End_R(M)$ commutes with all isomorphisms of $M$, then $M$ is an abelian endoregular module.
\end{prop}

\begin{proof}
Let $N$ and $K$ be two nonzero $M$-generated submodules of $M$ and let $\theta:N\to K$ be an isomorphism. Since $N$ is $M$-generated, $\Hom_R(M,N)\neq 0$. Let $0\neq\varphi:M\to N$. Set $A=\varphi(M)\leq N$ and $B=\theta(A)\leq K$. Hence, $A\cong B$. Since $A$ is finitely $M$-generated, $A\dleq M$. Therefore, there exists an idempotent $e\in \End_R(M)$ such that $A=eM$. Analogously there exists an idempotent $f\in\End_R(M)$ such that $B=fM$. Since any unit endoregular module satisfy the internal cancellation property, then $u^{-1}eu=f$ for some isomorphism $u \in \End_R(M)$ (See, \cite[Remark 12]{leeunit}). By assumption,
$$e-f=e-u^{-1}eu=u^{-1}ue-u^{-1}eu=u^{-1}(ue-eu)=0.$$
Thus, $e=f$. This implies $\varphi(M)=A=eM=fM=B\leq K$ and so $N=\Hom_R(M,N)M\leq K$. Analogously, $K\leq N$. Thus $N=K$. By Proposition \ref{propinc}, $M$ is an abelian endoregular module.
\end{proof}

\begin{rem}
If $M$ is an abelian endoregular module, then $M$ has SIP and SSP (Lemma \ref{endosip}). That is the set of direct summands (finitely $M$-generated submodules) of $M$ is a lattice. Moreover this lattice is complemented and distributive by Proposition \ref{propinc}, in other words it is a Boolean algebra.
\end{rem}

%\begin{proof}
%We have that the lattice of direct summands of $M$ is distributive. Now, if $N\oplus L=M=N\oplus K$ then $L\cong K$. By Proposition \ref{propinc}.(b) $K=L$.
%\end{proof}

%\textcolor{blue}{I think this corollary is not important, wht do you think?}
%\begin{cor}
%Let $R$ be an abelian regular ring. If $c\in R$ is a regular element then $c$ is a unit.
%\end{cor}
%
%\begin{proof}
%If $c$ is regular then $cR\cong R$. By Proposition \ref{propinc}, $R=cR$. Analogously $Rc=R$.
%\end{proof}

\begin{lemma}\label{fiinsub}
Let $M$ be a module and $L$ be a fully invariant direct summand of $M$. If $L\leq N\leq M$ then $L$ is fully invariant in $N$.
\end{lemma}

\begin{proof}
Let $N$ be a submodule of $M$ containing $L$ and $\alpha:N\to N$ be any endomorphism of $N$. Since $L$ is a direct summand of $M$ there exists a left semicentral idempotent $e\in\End_R(M)$ such that $eM=L$. Since $L\subseteq N$ we can consider $e:M\to N$. Let $\iota:N\to M$ denote the canonical inclusion. Hence $\iota\alpha e:M\to M$. Since $e$ is a left semicentral idempotent, $\iota\alpha e=(\iota\alpha e)e=e(\iota\alpha e)e$. Thus
\[\alpha(L)  =\iota\alpha(L)
 =\iota\alpha e(M)
 =e(\iota\alpha e)e(M)
 \subseteq e(M)
 = L.\]
Therefore, $L$ is fully invariant in $N$.
\end{proof}

%\begin{defn}
%Let $(P,\leq)$ be a poset. It is said $P$ is \emph{directed} if for any $x,y\in P$ there exists $z\in P$ such that $x\leq z$ and $y\leq z$.
%\end{defn}

\begin{thm}\label{limit}
Let $M$ be an abelian endoregular module and $N$ be an $M$-generated submodule of $M$. Then $N$ is an abelian endoregular module.
\end{thm}

\begin{proof}
Let $I=\Hom_R(M,N)$. Then $N=\sum\{f(M)\mid f\in I\}$. Since $M$ is an endoregular module, $f(M)\dleq M$ for every $f\in I$, moreover $f(M)$ is fully invariant in $N$ by Lemma \ref{fiinsub}. Let $f,g\in I$. Since $M$ is endoregular, there exists an idempotent $e\in S$ such that $eM=f(M)+g(M)$ (Lemma \ref{endosip}). Note that $f(M)+g(M)\leq N$, hence $e\in I$. Therefore, $\mathcal{J}:=\{f(M)\mid f\in I\}$ is a directed poset. We have that $f(M)$ is fully invariant direct summand of $M$, so if $f(M)\leq g(M)$ then $f(M)$ is fully invariant in $g(M)$ by Lemma \ref{fiinsub}. Hence, there exists a ring homomorphism $\End_R(g(M))\to \End_R(f(M))$ given by restriction. We claim that
\[\varprojlim_{f(M)\in \mathcal{J}}\End_R(f(M))\cong \End_R(N).\]
Let $A$ be a ring with ring homomorphisms $\{\pi_f:A\to \End_R(f(M))\mid f(M)\in \mathcal{J}\}$. Consider the following diagram:
\[\xymatrix{ & A\ar[dddl]_{\pi_g}\ar[dddr]^{\pi_f}\ar@{--{>}}[dd]^\lambda & \\ & & \\ & \End_R(N)\ar[dl]^{\mid_g}\ar[dr]_{\mid_f} & \\ \End(g(M))\ar[rr]_{\mid} & & \End_R(f(M))}\]
Suppose that $\mid\circ\pi_g=\pi_f$ whenever $f(M)\leq g(M)$. Let $a\in A$ and $\sum_{i=1}^\ell f_i(m_i)\in N=IM$. Define $\lambda(a):N\to N$ as
\[\lambda(a)(n)=\lambda(a)\left(\sum_{i=1}^\ell f_i(m_i)\right)=\sum_{i=1}^\ell \pi_{f_i}(a)\left(f_i(m_i)\right)\in N.\]
Suppose $\sum_{i=1}^\ell f_i(m_i)=0$. Since $\mathcal{J}$ is directed, there exists $g\in I$ such that $f_i(M)\subseteq g(M)$ for all $1\leq i\leq \ell$ and so $0=\sum_{i=1}^\ell f_i(m_i)=g(m)$ for some $m\in M$. Then
\[0  =\pi_g(a)(g(m))
 =\pi_g(a)\left(\sum_{i=1}^\ell f_i(m_i)\right)
 =\sum_{i=1}^\ell \pi_g(a)(f_i(m_i))
 =\sum_{i=1}^\ell \mid\circ\pi_g(a)(f_i(m_i))\]
\[ =\sum_{i=1}^\ell \pi_{f_i}(a)(f_i(m_i))
 =\lambda(a)\left(\sum_{i=1}^\ell f_i(m_i)\right).\]
Then, $\lambda(a)$ is well defined for all $a\in A$. It is clear that $\lambda$ is a homomorphism of abelian groups. Let $a,b\in A$ and $n=\sum_{i=1}^\ell f_i(m_i)\in N$ be arbitrary.  Since $\mathcal{J}$ is directed, there exists $g\in I$ such that $f_i(M)\subseteq g(M)$ for all $1\leq i\leq \ell$ and so $n=\sum_{i=1}^\ell f_i(m_i)=g(m)$ for some $m\in M$. Then
\[\lambda(a)\lambda(b)(n) =\lambda(a)\lambda(b)(g(m))
 =\lambda(a)\left(\pi_g(b)(g(m))\right)
 =\pi_g(a)\pi_g(b)(g(m))
 =\pi_g(ab)(g(m))\]
\[ =\lambda(ab)(g(m))
 =\lambda(ab)(n).\]

Thus, $\lambda$ is a ring homomorphism and it is clear that $\mid_g\circ\lambda=\pi_g$. Now, if there exists another ring homomorphism $\eta:A\to \End_R(N)$ such that $\mid_g\circ\eta=\pi_g$ then
\[\lambda(a)(g(m))=\pi_g(a)(g(m))=\mid_g\circ\eta(a)(g(m))=\eta(a)(g(m)).\]
for all $a\in A$ and all $g(m)\in N$. Thus $\lambda$ is unique, proving our claim.

By Corollary \ref{dirsummand}, $\End_R(f(M))$ is an abelian regular ring. Thus $\End_R(N)$ is an abelian regular ring by \cite[Proposition 3.6]{goodearlneumann}. Hence, $N$ is an abelian endoregular module.
\end{proof}

We have that every direct summand of an abelian endoregular module is an abelian endoregular module (Corollary \ref{dirsummand}), but in general the direct sum of two abelian endoregular modules is not. As an example we can consider the endoregular $\mathbb{Z}$-module, $M=\mathbb{Q}\oplus\mathbb{Q}$ (see \cite[Corollary 3.15]{leemodules}). Note that $\mathbb{Q}\oplus 0$ is $M$-generated but is not fully invariant in $M$. Hence $M$ is not abelian endoregular (Proposition \ref{subfi}). Next proposition characterize the direct sums of abelian endoregular modules.

\begin{prop}\label{sums}
Let $\{M_i\}_{i\in I}$ be a family of modules. Then, $\bigoplus_{i\in I}M_i$ is abelian endoregular if and only if $M_i\leq_{fi}\bigoplus_{i\in I}M_i$ and $M_i$ is an abelian endoregular module for all $i\in I$.
\end{prop}

\begin{proof}
Suppose $\bigoplus_{i\in I}M_i$ is an abelian endoregular module. Hence, $M_i$ is an abelian endoregular module for all $i\in I$ by Corollary \ref{dirsummand}, and $M_i$ is fully invariant in $\bigoplus_{i\in I}M_i$ by Proposition \ref{subfi}. Conversely, since each $M_i$ is an endoregular module then $\bigoplus_{i\in I}M_i$ is endoregular by \cite[Proposition 3.20]{leemodules}. Let $\varphi=(\varphi_{ij}):\bigoplus_{i\in I}M_i\to \bigoplus_{i\in I}M_i$ be an endomorphism where $\varphi_{ij}:M_j\to M_i$. Since $M_i\leq_{fi}\bigoplus_{i\in I}M_i$ for all $i\in I$, $\Ker\varphi=\bigoplus_{i\in I}\Ker\varphi_{ii}$ and $\Img\varphi=\bigoplus_{i\in I}\Img\varphi_{ii}$. By Proposition \ref{phikerima} $M_i=\Ker\varphi_{ii}\oplus\Img\varphi_{ii}$. It follows that
\[\bigoplus_{i\in I}M_i =\bigoplus_{i\in I}\left(\Ker\varphi_{ii}\oplus\Img\varphi_{ii}\right)
 =\bigoplus_{i\in I}\Ker\varphi_{ii}\oplus\bigoplus_{i\in I}\Img\varphi_{ii}
 =\Ker\varphi\oplus\Img\varphi.\]
Thus, $\bigoplus_{i\in I}M_i$ is an abelian endoregular module by Proposition \ref{phikerima}.
\end{proof}

\begin{cor}
Let $M=\bigoplus_{i\in I}Q_i$ be a direct sum of nonisomorphic indecomposable nonsingular injective modules. Then $M$ is abelian endoregular.
\end{cor}

\begin{proof}
Let $\varphi:Q_i\to Q_j$ be a nonzero homomorphism. Since $Q_i$ is uniform and $Q_j$ is nonsingular, $\varphi$ is a monomorphism. Since $Q_i$ is injective and $Q_j$ is indecomposable, $\varphi$ is an isomorphism. Then $\Hom_R(Q_i,Q_j)=0$ if $i\neq j$ and $\End_R(Q_i)$ is a division ring for every $i\in I$. By  Proposition \ref{sums}, $M$ is abelian endoregular.
\end{proof}

\begin{cor}
Let $M$ be a torsion-free module over a commutative domain $R$. Then $M$ is abelian endoregular if and only if $M$ is $R$-isomorphic to the field of fractions of $R$.
\end{cor}

\begin{proof}
In \cite[Lemma 4.21]{leemodules} is proved that, under these conditions, $M$ is endoregular if and only if $M$ is $R$-isomorphic to a direct sum of copies of the field of fractions $Q$ of $R$. Therefore, $M$ is abelian endoregular if and only if $M\cong Q$ as $R$-modules by Proposition \ref{sums}.
\end{proof}

Recall that a module $M$ is said to be \emph{retractable} if $\Hom_R(M,N)\neq 0$ for each nonzero submodule $N$ of $M$. A module $M$ is \emph{polyform} if for any submodule $K$ of $M$ and any $0\neq \varphi:K\to M$, $\Ker\varphi$ is not essential in $K$ \cite{wisbauermodules}. A similar notion is that of $\mathcal{K}$-nonsigular modules. A module $M$ is \emph{$\mathcal{K}$-nonsingular} if for all $0\neq\varphi\in\End_R(M)$, $\Ker\varphi$ is not essential in $M$ \cite{rizvik}. It is a fact that every endoregular module is $\mathcal{K}$-nonsingular (see Proposition \ref{kerdirecsum}).

The concepts of polyform and $\mathcal{K}$-nonsingular modules generalize that of nonsingular ring. It is easy to see that every polyform module is $\mathcal{K}$-nonsingular. The converse is not true in general. For, consider the $\mathbb{Z}$-module $\mathbb{Q}\oplus\mathbb{Z}_p$ with $p$ a prime number. Note that $\mathbb{Q}\oplus\mathbb{Z}_p$ is abelian endoregular by Proposition \ref{sums} and so it is $\mathcal{K}$-nonsingular. The homomorphism $\varphi:\mathbb{Z}\oplus\mathbb{Z}_p\to \mathbb{Q}\oplus\mathbb{Z}_p$ given by $\varphi(n,\bar{m})=(0,\bar{n})$ has essential kernel, hence $\mathbb{Q}\oplus\mathbb{Z}_p$ is not polyform. However, there are cases when polyform and $\mathcal{K}$-nonsingular modules coincide. Let $\widehat{M}=\Hom_R(M,E(M))M$ be the trace of $M$ into its injective hull $E(M)$ in $R$-Mod, that is, $\widehat{M}$ is the greatest $M$-generated submodule of $E(M)$. In \cite{wisbauerfoundations} it is proved that $\widehat{M}$ is a quasi-injective module and it is an essential extension of $M$.  Moreover, polyform modules are characterized as those modules $M$ such that $\widehat{M}$ is endoregular.

\begin{lemma}[11.1, \cite{wisbauermodules}]\label{polyendo}
The following conditions are equivalent for a module $M$:
\begin{enumerate}[label=\emph{(\alph*)}]
\item $M$ is polyform.
\item $\widehat{M}$ is an endoregular module.
\end{enumerate}
\end{lemma}

\begin{lemma}\label{polyknon}
Consider the following conditions for a module $M$:
\begin{enumerate}[label=\emph{(\arabic*)}]
\item $E(M)$ is $\mathcal{K}$-nonsingular;
\item $\widehat{M}$ is $\mathcal{K}$-nonsingular.
\item $M$ is polyform.
\item $\widehat{M}$ is polyform.
\end{enumerate}
Then \emph{(1)}$\Rightarrow$\emph{(2)}$\Leftrightarrow$\emph{(3)}$\Leftrightarrow$\emph{(4)}. If $M$ is not contained in $\Ker\varphi$ for all $\varphi\in\End_R(E(M))$ then the four conditions are equivalent.
\end{lemma}

\begin{proof}
(1)$\Rightarrow$(2) Let $0\neq\varphi:\widehat{M}\to \widehat{M}$ be an endomorphism such that $\Ker\varphi\ess\widehat{M}$. Since $E(M)$ is injective, there exists $\overline{\varphi}:E(M)\to E(M)$ such that $\overline{\varphi}(\widehat{M})=\varphi(\widehat{M})$. Hence $\Ker\varphi\leq\Ker\overline{\varphi}$. This implies that $\Ker\overline{\varphi}\ess E(M)$. Thus $E(M)$ is not $\mathcal{K}$-nonsingular.

(2)$\Rightarrow$(3) Since $\widehat{M}$ is a $\mathcal{K}$-nonsingular quasi-injective module, $\End_R(\widehat{M})$ is a von Neumann regular ring \cite[22.1]{wisbauerfoundations}. It follows from Lemma \ref{polyendo} that $M$ is polyform.

(3)$\Rightarrow$(4) Just note that $\widehat{\widehat{M}}=\widehat{M}$. Hence by Lemma \ref{polyendo} $\widehat{M}$ is polyform.

(4)$\Rightarrow$(2) It follows from the fact that every polyform module is $\mathcal{K}$-nonsingular.
%Let $N\leq\widehat{M}$ and $0\neq\varphi:N\to\widehat{M}$ be a homomorphism such that $\Ker\varphi\ess N$. Hence $0\neq L:= \varphi^{-1}(\varphi(N)\cap M)$ and $\Ker\varphi\cap M\ess M$. Therefore $0\neq\varphi|_L:L\cap M\to M$ and $\Ker\varphi|_L=\Ker\varphi\cap M$. Thus $M$ is not polyform.
\vspace{5pt}

Now, assume that $M$ is not contained in $\Ker\varphi$ for all $0\neq \varphi\in\End_R(E(M))$.

\noindent (4)$\Rightarrow$(1) Let $0\neq \varphi:E(M)\to E(M)$ such that $\Ker\varphi\ess E(M)$. Hence $0\neq \Ker\varphi\cap M\ess M$. By hypothesis $\varphi|_M\neq 0$. Note that $\varphi|_M:M\to \widehat{M}$ and $\Ker\varphi|_M=\Ker\varphi\cap M$. But $\Ker\varphi|_M\ess M$ which is a contradiction because $\widehat{M}$ is polyform. Thus $E(M)$ is $\mathcal{K}$-nonsingular.
\end{proof}

\begin{rem}
It is not difficult to see that for a nonsingular module, the four conditions in Lemma \ref{polyknon} are equivalent.
\end{rem}

Recall that a submodule $N$ of a module $M$ is said to be \emph{essentially closed} if $N\ess L\leq M$ then $L=N$.
\begin{prop}\label{emsen}
Let $M$ be a module. If $E(M)$ is abelian endoregular then $M$ is polyform, every essentially closed submodule of $M$ is fully invariant and $\End_R(M)$ is a reduced ring.
\end{prop}

\begin{proof}
Since $E(M)$ is endoregular, $E(M)$ is $\mathcal{K}$-nonsingular and by Lemma \ref{polyknon} $M$ is polyform. Moreover, $S=\End_R(M)$ can be embedded in $T=\End_R(E(M))$ as subring. This implies that $S$ is a reduced ring. Let $N\leq M$ be an essentially closed submodule. There is a decomposition $E(M)=K\oplus L$ such that $N\ess K$. Then $N\ess K\cap M\subseteq M$. This implies that $N=K\cap M$. On the other hand, since $K\dleq E(M)$, $K$ is fully invariant in $E(M)$ (Proposition \ref{subfi}). Hence $N=K\cap M$ is fully invariant in $M$.
\end{proof}

\begin{thm}\label{injhullsr}
Consider the following conditions for a module $M$:
\begin{enumerate}[label=\emph{(\arabic*)}]
\item $\widehat{M}$ is an abelian endoregular module;
\item $M$ is polyform, every essentially closed submodule of $M$ is fully invariant and $\End_R(M)$ is a reduced ring.
\end{enumerate}
Then \emph{(1)}$\Rightarrow$\emph{(2)}. Moreover, if $M$ is retractable then the converse holds.
\end{thm}

\begin{proof}
(1)$\Rightarrow$(2) This is similar to the proof of Proposition \ref{emsen}.
%Since $\widehat{M}$ is endoregular, $M$ is polyform. Then $S=\End_R(M)$ can be embedded in $T=\End_R(\widehat{M})$ as subring. This implies that $S$ is a reduced ring. Let $N\leq M$ be an essentially closed submodule. There is a decomposition $\widehat{M}=K\oplus L$ such that $N\ess K$. Then $N\ess K\cap M\subseteq M$. This implies that $N=K\cap M$. On the other hand, since $K\dleq \widehat{M}$, $K$ is fully invariant in $\widehat{M}$. Hence $N=K\cap M$ is fully invariant in $M$.

Now suppose $M$ is retractable. (2)$\Rightarrow$(1) Since $M$ is polyform and retractable, $S=\End_R(M)$ can be embedded in $T=\End_R(\widehat{M})$ as subring and $S_S$ is essential in $T_S$ \cite[11.5]{wisbauermodules}. By Lemma \ref{polyendo}, $T$ is a von Neumann regular ring. Let $0\neq\varphi\in T$ such that $\varphi^2=0$. Since $S_S\ess T_S$, there exists $\psi\in T$ with $\psi(M)\subseteq M$ such that $0\neq\varphi \psi(M)\subseteq M$. On the other hand, $\varphi \psi(\widehat{M})\dleq \widehat{M}$ because $\widehat{M}$ is endoregular. Since $M$ is polyform, $\varphi \psi(\widehat{M})\cap M$ is essentially closed in $M$. Therefore, $\varphi\psi(\widehat{M})\cap M$ is fully invariant in $M$. Hence
\[\psi\varphi\psi(M)\subseteq \psi(\varphi\psi(\widehat{M})\cap M)\subseteq \varphi\psi(\widehat{M})\cap M.\]
This implies that $(\varphi\psi)^2(M)=0$. Since $S$ is reduced, $\varphi\psi=0$, contradiction. Thus $\varphi=0$. Therefore $T$ is an abelian regular ring because $T$ is regular and has no nonzero nilpotent elements \cite[Theorem 3.2]{goodearlneumann}.
\end{proof}

The following corollary characterizes those rings whose maximal ring of quotients is abelian regular. The characterization of these rings is due to Renault and Utumi (see \cite[Ch.XII, Proposition 5.2]{stenstromrings}). Note that for a ring $R$, we have that $\widehat{R}=E(R)$ and when $R$ is nonsingular $Q_{max}(R)\cong\End_R(E(R))$.

\begin{cor}
The following conditions are equivalent for a ring $R$:
\begin{enumerate}[label=\emph{(\alph*)}]
\item $Q_{max}(R)$ is an abelian regular ring.
\item $R$ is a nonsingular reduced ring and every essentially closed right ideal is two-sided.
\end{enumerate}
\end{cor}

\begin{cor}
If $R$ is an abelian regular ring then $Q_{max}(R)$ is abelian regular.
\end{cor}

%\begin{defn}
%A ring $R$ is a right \emph{V-ring} if every right ideal is an intersection of maximal right ideals.
%\end{defn}
%
%\begin{prop}[\cite{wisbauerfoundations}, Ex. 23.9.(2)]
%The following conditions are equivalent for a ring $R$:
%\begin{enumerate}
%\item $R$ is an abelian regular ring.
%\item $R$ is a right V-ring and every (maximal) right ideal is an ideal.
%\end{enumerate}
%\end{prop}
%
%What is an analogous result in the module case???

\section{Abelian endoregular modules as subdirect products}

Recall that a module $M$ is said to be a \emph{subdirect product} of the family $\{N_i\}_\mathcal{I}$ if $M$ is a submodule of $\prod_\mathcal{I} N_i$ such that the canonical projections restricted to $M$ are surjective. In some results we will use the concept of (quasi-)duo modules. Recall that a module $M$ is (\emph{quasi}-)\emph{duo} if every (maximal) submodule is fully invariant.

\begin{prop}\label{sdpse}
Let $M$ be a quasi-duo module. If $M$ is a subdirect product of simple modules then the idempotents in $\End_R(M)$ are central. In particular, if $M$ is endoregular then $M$ is abelian endoregular.
\end{prop}

\begin{proof}
By hypothesis there exists a monomorphism $\alpha:M\to \prod_{j\in \mathcal{J}}S_j$ with $S_j$ simple, such that $\pi_j\alpha$ are surjective for every $j\in \mathcal{J}$ where $\pi_j$ are the canonical projections. Hence $\Ker\pi_j\alpha$ is a maximal submodule of $M$. Since $M$ is quasi-duo, $\Ker\pi_j\alpha$ is fully invariant. Let $\varphi:M\to M$ be any endomorphism. Since $\Ker\pi_j\alpha$ is fully invariant in $M$, $\varphi$ induces an endomorphism $\varphi_j:S_j\to S_j$ such that $\varphi_j\pi_j\alpha=\pi_j\alpha\varphi$.
%Suppose there exists $\psi:S_j\to S_j$ such that $\psi\pi_j\alpha=\pi_j\alpha\varphi$. Let $0\neq x\in S_j$, by hypothesis there exists $m\in M$ such that $\pi_j\alpha(m)=x$. Hence
%\[\psi(x)=\psi\pi_j\alpha(m)=\pi_j\alpha\varphi(m)=\varphi_j\pi_j\alpha(m)=\varphi_j(x).\]
%Therefore, $\psi=\varphi_j$.
This defines a ring homomorphism
\[\Theta:\End_R(M)\to \prod_{j\in\mathcal{J}}\End_R(S_j)\]
given by $\Theta(\varphi)=(\varphi_j)$. Suppose $\Theta(\varphi)=0$, i.e., $\varphi_j=0$ for all $j\in\mathcal{J}$. Then $\pi_j\alpha\varphi=0$ for all $j\in \mathcal{J}$. This implies that $\alpha\varphi=0$. Since $\alpha$ is a monomorphism, $\varphi=0$. Thus $\Theta$ is a monomorphism. Hence, $\End_R(M)$ is a subring of a direct product of division rings. It follows that every idempotent in $\End_R(M)$ is central. As a consequence, if $M$ is endoregular then $M$ is an abelian endoregular module.
\end{proof}

The following example shows that the hypothesis on $M$ of being a quasi-duo module is not superfluous.

\begin{example}
Let $M=\mathbb{R}\oplus \mathbb{R}$. Hence $M$ is a subdirect product of simple modules but $M$ is not quasi-duo. We have that $\End_\mathbb{R}(M)=\Mat_2(\mathbb{R})$ which is a von Neumann regular ring but it is not abelian regular. Hence $M$ is an endoregular module but it is not abelian endoregular.
\end{example}

\begin{cor}
Let $M$ be an endoregular quasi-duo module. If $\Rad(M)=0$ then $M$ is an abelian endoregular module.
\end{cor}

A submodule of an $M$-generated module is called \emph{$M$-subgenerated}. The full subcategory of $R$-Mod consisting of all $M$-subgenerated modules is denoted by $\sm$. Note that, if $M=R$ then $\sm=R$-Mod. This category is a Grothendieck category as is shown in \cite{wisbauerfoundations}. In \cite[18.3]{wisbauerfoundations} is proved that $M$ is $M^{(\mathcal{I})}$-projective for every index set $\mathcal{I}$ if and only if $M$ is a projective object in $\sm$; and a finitely generated module $M$ is projective in $\sm$ if and only if $M$ is $M$-projective (quasi-projective). A module $M$ is a \emph{generator} of $\sm$ if every module in $\sm$ is $M$-generated. It is not difficult to see that if a module is semiprime (Definition \ref{defsemiprime}) and projective in $\sm$ then $M$ is retractable \cite[Lemma 1.24]{maugoldie}.

\begin{lemma}\label{primmax}
Let $M$ be $M^{(\mathcal{I})}$-projective for every index set $\mathcal{I}$. If $M$ is an abelian endoregular module then every prime submodule is a maximal submodule. Moreover $M$ is a quasi-duo module.
\end{lemma}

\begin{proof}
Let $P\leq M$ be prime in $M$. Since $P$ is fully invariant, every endomorphism $\varphi:M\to M$ defines an endomorphism $\overline{\varphi}:M/P\to M/P$. Then, there exists a ring homomorphism
\[\Theta:\End_R(M)\to \End_R(M/P).\]
Note that $\Theta$ is surjective because $M$ is quasi-projective. Since $M$ is projective in $\sm$, $M/P$ is projective in $\sigma[M/P]$ by \cite[Lemma 9]{vanmodules}. We have $M/P$ is a prime module, hence $\End_R(M/P)$ is a prime ring by \cite[Lemma 5.9]{mauozcan}. This implies that $\Ker\Theta$ is a prime ideal of $\End_R(M)$. Therefore $\End_R(M/P)$ is a division ring by \cite[Thorem 3.2]{goodearlneumann}. Since $M/P$ is retractable, $M/P$ is a simple module by \cite[Proposition 4.12]{leemodules}. Thus, $P$ is a maximal submodule.

Let $\mathcal{M}$ be a maximal submodule of $M$. Then $P=\bigcap\{\Ker\varphi\mid\varphi\in\Hom_R(M,M/\mathcal{M})\}$ is prime in $M$ by \cite[Proposition 3.4]{mauozcan} and $P\leq\mathcal{M}$. But $P$ is a maximal submodule. Hence, $\mathcal{M}=P$. Thus, $M$ is quasi-duo.
\end{proof}

\begin{rem}\label{3}
From last proposition and Lemma \ref{maxprime}, we can see that for an abelian endoregular module $M$ such that $M$ is $M^{(\mathcal{I})}$-projective for every index set $\mathcal{I}$, a submodule $P$ of $M$ is prime in $M$ if and only if $P$ is a maximal submodule.
\end{rem}

\begin{thm}\label{stqdr}
Consider the following conditions for a module $M$:
\begin{enumerate}[label=\emph{(\arabic*)}]
\item $M$ is a quasi-duo endoregular module and $M$ is a subdirect product of simple modules.
\item $M$ is an abelian endoregular module and $Rad(M)=0$.
\end{enumerate}
Then \emph{(1)}$\Rightarrow$\emph{(2)}. If $M$ is $M^{(\mathcal{I})}$-projective for every index set $\mathcal{I}$ then the two conditions are equivalent.
\end{thm}

\begin{proof}
(1)$\Rightarrow$(2) Suppose there exists a monomorphism $\alpha:M\to \prod_{j\in \mathcal{J}}S_j$ with $S_j$ simple, such that $\pi_j\alpha$ are surjective for every $j\in \mathcal{J}$ where $\pi_j$ are the canonical projections. Note that $0=\Ker\alpha=\bigcap_{j\in\mathcal{J}}\Ker\pi_j\alpha$ and each $\Ker\pi_j\alpha$ is a maximal submodule of $M$. Therefore $Rad(M)\subseteq\bigcap_{j\in\mathcal{J}}\Ker\pi_j\alpha=0$. It follows from Proposition \ref{sdpse}, $M$ is an abelian endoregular module.

Now suppose that $M$ is $M^{(\mathcal{I})}$-projective for every index set $\mathcal{I}$.

\noindent (2)$\Rightarrow$(1) By Lemma \ref{primmax}, $M$ is quasi-duo. Since $Rad(M)=0$, the canonical homomorphism
\[\alpha:M\to \prod_{P\in Spec(M)}M/P\]
is a monomorphism (Remark \ref{3}). By Lemma \ref{primmax}, each $M/P$ is a simple module. Let $\pi_P$ denote the canonical projections. It is clear that $\pi_P\alpha$ is surjective for all $P\in Spec(M)$.
\end{proof}

The implication (2)$\Rightarrow$(1) in Theorem \ref{stqdr} may not be true if we do not impose the projectivity condition on $M$ or if $Rad(M)\neq 0$ as the following examples show.

\begin{example}
\begin{enumerate}
\item[(i)] Consider $\mathbb{Q}_\mathbb{Z}$. It can be seen that $\mathbb{Q}_\mathbb{Z}$ is not quasi-projective. On the other hand, $\mathbb{Q}_\mathbb{Z}$ is an abelian endoregular module. Note that $\mathbb{Q}_\mathbb{Z}$ is not a subdirect product of simple modules.

\item[(ii)] Let $R= \left(\begin{smallmatrix}
 \mathbb{Z}_{2}& \mathbb{Z}_{2} \\
          0 & \mathbb{Z}_{2}
\end{smallmatrix}\right)$ and $e_1=\left(\begin{smallmatrix}
 1 & 0 \\
        0 & 0
\end{smallmatrix}\right)\in R$. Consider the module $M=e_{1}R=\left(\begin{smallmatrix}
                \mathbb{Z}_{2}& \mathbb{Z}_{2} \\
                            0 & 0
\end{smallmatrix}\right)$. Note that $Rad(M)\neq 0$.
Since $\End_R(M)\cong \mathbb{Z}_2$, $M$ is a projective abelian endoregular module.
However $M$ is not a subdirect product of simple modules.
%semiprime as
%\[\begin{pmatrix}
%              0 & \mathbb{Z}_{2} \\
%              0 & 0
%            \end{pmatrix}
%_M\begin{pmatrix}
%              0 & \mathbb{Z}_{2} \\
%              0 & 0
%            \end{pmatrix}
%=0.\]

\end{enumerate}\end{example}

\begin{cor}
The following conditions are equivalent for a von Neumann regular ring $R$:
\begin{enumerate}[label=\emph{(\alph*)}]
\item $R$ is an abelian regular ring.
\item $R$ is a quasi-duo ring and $R$ is a subdirect product of simple $R$-modules.
\end{enumerate}
\end{cor}

Recall that a module $M$ is said to be \emph{co-semisimple} if each simple module (in $\sm$) is $M$-injective. In the ring case, this concept is known as right \emph{V-ring}. In \cite{garciaendomorphism} the authors present some results about modules whose endormorphism rings are V-rings.

\begin{prop}\label{coss}
Let $M$ be a finitely generated quasi-projective module. The following conditions are equivalent:
\begin{enumerate}[label=\emph{(\alph*)}]
\item $M$ is abelian endoregular and a generator of $\sm$.
\item $M$ is co-semisimple and quasi-duo.
\item $M$ is duo and a generator of $\sm$ such that every cyclic submodule is a direct summand.
\end{enumerate}
\end{prop}

\begin{proof}
(a)$\Rightarrow$(b) Since every abelian regular ring is a V-ring, by \cite[Corollary 2.4]{garciaendomorphism} $M$ is co-semisimple. By Lemma \ref{primmax} $M$ is quasi-duo.

(b)$\Leftrightarrow$(c) It follows from \cite[Proposition 5.13]{mauozcan}.

(c)$\Rightarrow$(a) We claim that $S=\End_R(M)$ is a von Neumann regular ring. Let $f\in S$. We have $fSM=f(M)$. Since $f(M)$ is finitely generated $f(M)\dleq M$ \cite[Remark 6.1]{tuganbaevrings}, that is, $M=fSM\oplus N$ for some $N\leq M$. On the other hand, $fS=\Hom_R(M,fSM)$ because  $M$ is quasi-projective. Hence,
\[S=\Hom_R(M,fSM)\oplus\Hom_R(M,N)=fS\oplus \Hom_R(M,N).\]
Thus $S$ is a von Neumann regular ring, that is, $M$ is endoregular. Since $M$ is duo, $M$ is abelian endoregular by Proposition \ref{subfi}.
\end{proof}

\begin{cor}
The following conditions are equivalent for a ring $R$:
\begin{enumerate}[label=\emph{(\alph*)}]
\item $R$ is an abelian regular ring.
\item $R$ is a quasi-duo $V$-ring.
\item $R$ is a duo von Neumann regular ring.
\end{enumerate}
\end{cor}

\section*{Acknowledgment}

The authors are very thankful to Prof. Gangyong Lee for suggesting to study this subject and for his remarks on this work.

\bibliographystyle{acm}
\bibliography{biblio}

\end{document}